\documentclass[a4paper]{amsart}
\DeclareMathOperator{\rank}{rank}
\DeclareMathOperator{\Gal}{Gal}
\DeclareMathOperator{\red}{red}
\usepackage[T1]{fontenc}
\newcommand{\Q}{\mathbf{Q}}
\newcommand{\Z}{\mathbf{Z}}

\newtheorem{theorem}{Theorem}[section]
\newtheorem{lemma}[theorem]{Lemma}
\newtheorem{proposition}[theorem]{Proposition}

\theoremstyle{definition}

\theoremstyle{remark}
\newtheorem{remark}[theorem]{Remark}
\title[Mordell--Weil group of a rational elliptic surfaces]{Determining explicitly the Mordell--Weil group of  certain rational elliptic surfaces}
\author{Remke Kloosterman}
\thanks{The author would like to thank the anonymous referees for several comments on a previous version of this paper.
The author is a member  of INdAM-GNSAGA. This work was supported  by the UNIPD BIRD-SID-2024 
project ``Moduli spaces and positivity".}
\begin{document}
\begin{abstract}

Let $A,B$ be nonzero rational numbers. Consider the elliptic curve
$E_{A,B}/\Q(t)$ with Weierstrass equation $y^2=x^3+At^6+B$.

An algorithm to determine $\rank E_{A,B}(\Q(t))$ as a function of $(A,B)$ was presented in a recent paper \cite{Bartosz}.
We will give a different and shorter proof for the correctness of that algorithm, using a more geometric approach and discuss for which classes of examples the approach by \cite{Bartosz} might be useful. 
\end{abstract}
\maketitle

\section{Introduction}
Let $A,B$ be nonzero rational numbers. Consider the elliptic curve
$E_{A,B}/\Q(t)$ with Weierstrass equation $y^2=x^3+At^6+B$.

In this paper we present a short proof for the correctness of the algorithm from \cite{Bartosz} by Desjardins and Naskr\k{e}cki
 to calculate $\rank E_{A,B}(\Q(t))$ as a function of $A$ and $B$. 
The originally  proof is quite lengthy, it is the heart of a 44-page paper and uses heavy computations.
We will give a different and shorter proof for the correctness of that algorithm, using a more geometric approach. This approach is similar to the approaches taken in \cite{ToMe} and \cite{KloKuw} to study the Mordell--Weil group of elliptic surfaces with several automorphisms.

Recall that the Weierstrass equation for curve $E_{A,B}/\Q(t)$ defines also a rational elliptic surface $S$ over $\Q$. 
The main idea is to establish four further rational elliptic surfaces, each having geometric Mordell--Weil rank 2 and each admitting a base change which is birational to $S$. We then show that the direct sum of the four Mordell--Weil groups has finite index in the Mordell--Weil group of $E_{A,B}$. It turns out that each of the four groups can be calculated easily.

\section{Calculation of the Mordell--Weil group}
Fix $A,B\in \Q^*$. Let $E_{A,B}/\Q(t)$ be the elliptic curve given by $y^2=x^3+At^6+B$. 
Let $K$ be a field of characteristic zero. Let $\omega \in\overline{\Q}$ be a fixed primitive third root of unity.

\begin{proposition}  We have $E_{A,B}(K(t))\cong E_{B,A}(K(t))$ as abelian groups. Moreover, for every $u,v \in K^*$ we have
\[ E_{A,B}(K(t))\cong E_{u^6A,v^6B}(K(t))\]
as abelian groups.
\end{proposition}
\begin{proof}
For the first statement take $y=y'/t'^3, x=x'/t'^2,t=1/t'$ and substitute this in the equation for $E_{A,B}$. Then we obtain a curve $E'$ with Weierstrass equation
\[y'^2=x'^3+A+B(t')^6.\]
Hence $E_{B,A}(K(t'))\cong  E_{A,B}(K(t)).$
For the second statement take $y=y'/u^3,x=x'/u^2, t=t'v/u$ and substitute this in the equation for $E_{A,B}$. Then we obtain a curve $E'$ with Weierstrass equation
\[y'^2=x'^3+Av^6(t')^6+Bu^6.\]
Hence  $E_{u^6A,v^6B}(K(t'))\cong  E_{A,B}(K(t)).$
\end{proof}

Recall that $E_{A,B}$ has complex multiplication by $\sqrt{-3}$, that is, $E_{A,B}$ admits the endomorphism (actually automorphism) $\tau: (x,y) \mapsto (\omega x,-y)$. For historical reasons we use $\omega$ both to indicate a fixed primitive third root of unity in $\overline{K}$  and the automorphism $-\tau=\tau^4$ of order 3.
 If $\omega\in K$ then $E_{A,B}(K(t))$ is a $\Z[\omega]$-module. 

In the following we will determine the group structure of the geometric Mordell--Weil group $E_{A,B}(\overline{\Q}(t))$. This can be obtained by very standard techniques. I.e., the rank can be expressed in the number of additive and of multiplicative fibers and the torsion group can be deterimined from the singular fibers. This has been carried out by several authors in the literature. In our case we obtain the following result:

\begin{lemma} The $\Z[\omega]$-module $E_{A,B}(\overline{\Q}(t))$ is free of rank 4. In particular, the geometric Mordell--Weil rank equals 8.
\end{lemma}
\begin{proof}
The discriminant of $E_{A,B}$ equals a constant times $(At^6+B)^2$. Hence there are six singular fibers. Since the valuation of the discriminant is two at each singular fiber and the $j$-invariant is zero, we have six $II$-fibers. It follows now from  \cite[Theorem 8.8]{SchShi} that $E_{A,B}(\overline{\Q}(t))\cong \Z^8$ as abelian groups.

To show that $E(\overline{\Q}(t))$ is a free $\Z[\omega]$-module of rank 4, let $P\in E(\overline{\Q}(t)), P\neq O$ and $a,b\in \Z$ and suppose that $(a+b\omega)P=O$ holds in $E_{A,B}(\overline{\Q}(t))$. Then $(a^2+b^2-ab)P=(a+b\omega^2)(a+\omega b) P=O$. Since $E(\overline{\Q}(t))$ is torsion-free as $\Z$-module this implies that 
\[\left(a-\frac{1}{2}b\right)^2+\frac{3}{4}b^2=a^2+b^2-ab=0,\] from which it follows that $a=b=0$.
\end{proof}

We are now interested in the action of $\Gal(\overline{\Q}/\Q)$ on $E_{A,B}(\overline{\Q}(t))$. We will show that  this Galois-module is the direct sum of four Galois modules (each of $\Z$-rank 2). For this  we consider for a positive integer $m$ and a  integer $\/0\leq k\leq 5$   the elliptic curve $E_{A,B,k,m}/\Q(s)$ given by 
\[ y^2=x^3+s^k(As^m+B).\]
We will mainly focus on the case $m=1$ and $1\leq k \leq 4$.  
\begin{lemma}\label{lemrk} For $1\leq k \leq 4$ we have that
$E_{A,B,k,1}(\overline{\Q}(s))$ is a free  $\Z[\omega]$-module of rank one.
\end{lemma}
\begin{proof}
Note that the discriminant of $E_{A,B,k.1}$ equals $s^{2k}(As+B)^2$ up to a constant.
The degree of the discriminant is $2(k+1)\leq 10$, hence we have three singular fibers at the zeroes of the discriminant ($s=0$, $s=-B/A$, $s=\infty$). These fibers have in total six components not intersecting the zero section.
From the Shioda-Tate formula (see e.g. \cite[Theorem 8.8]{SchShi}) it follows that the geometric Mordell--Weil rank equals 2, hence this group has rank one as a $\Z[\omega]$-module.
\end{proof}

\begin{lemma} Suppose $0\leq k \leq 5$ and let $m$ be a positive integer dividing $\gcd(6,k)$. Then $E_{A,B,k,m}(K(s))$ is a subgroup of $E_{A,B}(K(t))$.
\end{lemma}
\begin{proof}
Let $q=6/m$, $r=k/m$. Then
\[(f(s),g(s))\mapsto \left(\frac{f(t^{q})}{t^{2r}},\frac{g(t^{q})}{t^{{3r}}}\right)  \]
defines an injective map $E_{A,B,k,m}(K(s))\to E_{A,B}(K(t))$.
\end{proof}

The pairs $(k,m)$ satisfying the hypothesis of this lemma are $m=1$ and $k\in \{0,1,2,3,4,5\}$;  $m=2$ and $k\in \{0,2,4\}$ and $m=3$ and $k\in \{0,3\}$.
We will show that four of the six curves with $m=1$ generate $E_{A,B}(K(t))$ up to finite index:

\begin{lemma}
We  have that
\[ \oplus_{k=1}^4 E_{A,B,k,1}(K(s))\]
is a subgroup of finite index of $E_{A,B}(K(t))$.
\end{lemma}
\begin{proof} Let $\sigma$ be the  $K$-linear automorphism of $K(t)$ sending $t\mapsto -\omega t$. Then $\sigma$ acts on the Galois-module $E_{A,B}(\overline{\Q}(t))$.
Now $\sigma^*$ acts on the image of $E_{A,B,k,1}$ by sending $(x,y)$ to  $(\omega^{k}x, (-1)^ky)$, i.e., $\sigma^*$ acts as $\tau^k$. In particular, the pullback of $E_{A,B,k,1}(K(s))$ is contained in  the kernel of $\sigma-\tau^{k}$. 
The intersection of two such kernels for different values of $k$ is zero, hence the sum is  a direct sum.
If $K$ is algebraically closed then both groups have rank $8$, hence it is a subgroup of finite index. If $K$ is not algebraically closed, then the Galois-representations are respected, hence the above statement holds  for any field between $K$ and $\overline{K}$.
\end{proof}

\begin{remark} Similarly one can show that the following inclusions of free $\Z[\omega]$-modules have finite index.
\[ \begin{array}{r}
E_{A,B,3,1}(K(s)) \subset E_{A,B,0,2}(K(s))\\
E_{A,B,1,1}(K(s))\oplus E_{A,B,4,1}(K(s))\subset E_{A,B,2,2}(K(s))\\
E_{A,B,2,1}(K(s))\subset E_{A,B,4,2}(K(s))\\
E_{A,B,2,1}(K(s))\oplus E_{A,B,4,1}(K(s))\subset E_{A,B,3,0}(K(s))\\
E_{A,B,1,1}(K(s))\oplus E_{A,B,3,1}(K(s))\subset E_{A,B,3,3}(K(s))\end{array}
\]
\end{remark}

The $\Z$-module $\oplus_{k=1}^4 E_{A,B,k,1}(\overline{\Q}(s))$ is torsion-free and has finite index in the abelian group $E_{A,B}(\overline{\Q}(t))$. Hence the rank of $E_{A,B}(\Q(t))$ is the sum of the ranks of the Galois invariant subgroups of $E_{A,B,k,1}(\overline{\Q}(s))$, for $k=1,2,3,4$. 
Moreover,  each of  $E_{A,B,k,1}(\overline{\Q}(s))$ has rank 2, and is a $\Z[\omega]$-module, hence the the invariant subgroups turn out to have rank 0 or 1.
It remains now to determine each of these two dimensional representations. For this we include two probably well-known lemmata.

\begin{lemma}\label{lemext} Suppose $(-3)$ is not a square in $K$. Let $L=K(\sqrt{-3})$. For every  $u\in K^*$  we have that  $u$ is a square in $L$ if and only if $u$ or $-3u$ is a square in $K$.
Similarly, we have that $u$ is a cube in $L$ if and only if $u$ is a cube in $K$.
\end{lemma}
\begin{proof}
Let $a,b\in K$ and consider  $v=(a+\sqrt{-3}b)\in L$ then $v^2=a^2-3b^2+\sqrt{-3}ab$. Hence $v^2\in K$ if and only if $ab=0$. Hence $v^2$ is already a square in $L$ or $v^2$ equals $-3$ times a square in $L$. 

Similarly, $v^3=a^3+3 \sqrt{-3} a^2b-9 ab^2-3\sqrt{-3} b^3$.
Hence $v^3\in K$ if and only if $3ba^2+(-3)b^3=0$. Hence $b=0$ or $b^2=a^2$. In both cases $v^3$ is a cube in $K$.\end{proof}

\begin{lemma}\label{leminv}
Let $M$ be a free $\Z[\omega]$-module of rank 1 together with an action of
$G:=\Gal(\overline{K}/K)$, such that for all $m\in M,g\in G$ we have $g(\omega m)=g(\omega)g(m)$.

If $\omega \in K$ then $M^G\neq 0$ if and only if $M^G=M$. 

If $\omega \not \in K$ then let $H= \Gal(\overline{K}/K(\omega))$. %
If $M^H$ is nonzero then $M^G$ is nonzero and is a free abelian group of rank 1.
\end{lemma}

\begin{proof}
The case $\omega \in K$ is straightforward. If $\omega \not \in K$ and $M^H$ nonzero then $M^H=M$ by the previous case.

Let $\sigma$ be a generator of $\Gal(K(\omega)/K)$. Then  $\sigma^*:M\otimes_{\Z}\Q \to M \otimes_{\Z}\Q$ has order two, hence is diagonalizable and its eigenvalues are $\pm 1$.

Since $\sigma(\omega P)=\sigma(\omega) \sigma(P)$ by assumption we obtain that $\sigma^*$ has two distinct eigenvalues on this two-dimensional vector space.
Hence the $1$ eigenspace is one-dimensional and therefore $M^G$ is free of $\Z$-rank one.
\end{proof}

In order to find a generator for the geometric Mordell--Weil group of $E_{A,B,k,1}$ it suffices to write $As^{k+1}+Bs^k$ as a sum of a square and a cube in $\overline{\Q}[t]$. For $k=2$ this is automatic, for $k=1$ this can be done by completing the square.
We make this precise in the following proposition. 

\begin{proposition} Let $R=\Z$ if $\omega \not \in K$ and let $R=\Z[\omega]$ if $\omega \in K$. 
For $k\in \{1,2,3,4\}$ we have that 
$E_{A,B,k,1}(K(s))$ is nontrivial if and only if
\begin{enumerate}
\item $k=1$; $4AB\in (K^*)^3$ and  $A\in (K^*)^2$ or $-3A \in (K^*)^2$.
\item $k=2$; $A\in (K^*)^3$ and $B\in (K^*)^2$ or $-3B \in (K^*)^2$.
\item $k=3$; $B\in (K^*)^3$ and $A\in (K^*)^2$ or $-3A \in (K^*)^2$.
\item $k=4$; $4AB\in (K^*)^3$ and $B\in (K^*)^2$ or $-3B \in (K^*)^2$.
\end{enumerate}
Moreover, if this group is nontrivial then it is a free $R$-module of rank 1.
\end{proposition}
\begin{proof}
From Lemma~\ref{lemrk} it follows that $E_{A,B,k,1}(\overline{\Q}(s))$ is a free $\Z[\omega]$-module of rank 1. We start by determining a generator of a submodule of finite index.

Suppose first that $k=1$. Let $\alpha\in \overline{\Q}$ be a square root of $A$, let $\gamma \in \overline{\Q}$ be a third root of $\frac{B^2}{4A}$.
Then the (free, rank one) $\Z[\omega]$-module $E_{A,B,1,1}(\overline{\Q}(s))$ is generated up to finite index by
  \[P:=(\gamma,\alpha(s+B/2A)).\]
 Hence if $E_{A,B,1,1}(K(s))$ is non-trivial then $E_{A,B,1,1}(K(\omega)(s))=E_{A,B,1,1}(\overline{\Q}(s))$. Therefore $P\in E_{A,B,1,1}(K(\omega)(s))$. In particular, $A$ is a square in $K(\omega)$  and $4AB$ is a cube in $K( \omega)$. This is equivalent to $A$ or $-3A$ being a square in $K$ and $4AB$ a cube in $K$. (If $\omega \in K$ then $-3$ is a square, otherwise apply Lemma~\ref{lemext}.)
 
 For the converse direction, note that if $4AB\in (K^*)^3$ and $A$ or $-3A$ is in $(K^*)^2$, then $P\in E(K(\omega)(s))$. Lemma~\ref{leminv} implies now that $R$-rank of $E(K(s))$ equals one.
    
The case $k=4$ can be obtained from $k=1$ by applying the field automorphism $s\mapsto \frac{1}{s}$ and  interchanging the role of $A$ and $B$.
The case $k=2$ can be obtained  by taking  $\alpha$ be a third root of $A$, $\beta$ be a square root of $B$ and using the point
$(-\alpha s,\beta s)$ in $E_{A,B,2,1}(\overline{\Q}(s))$. The case $k=3$ can be obtained from the case $k=2$ by applying the field automorphism $s\mapsto \frac{1}{s}$ and  interchanging the role of $A$ and $B$.\end{proof}

Collecting everything we obtain the following result:
\begin{proposition} Define integers $r_1,\dots,r_4\in \{0,1\}$ as follows
\begin{itemize}
\item  $r_1=1$ if $4AB\in (\Q^*)^3$ and  $A\in(\Q^*)^2$ or $-3A\in(\Q^*)^2$; $r_1=0$ otherwise.
\item  $r_2=1$ if $A\in (\Q^*)^3$  and $B\in(\Q^*)^2$ or $-3B\in(\Q^*)^2$; $r_2=0$ otherwise.
\item  $r_3=1$ if $B\in (\Q^*)^3$  and  $A\in(\Q^*)^2$ or $-3A\in(\Q^*)^2$; $r_3=0$ otherwise.
\item  $r_4=1$ if $4AB\in (\Q^*)^3$  and $B\in(\Q^*)^2$ or $-3B\in(\Q^*)^2$; $r_4=0$ otherwise.\end{itemize}
Then
\[ \rank E_{A,B}(\Q(t))=r_1+r_2+r_3+r_4.\]
\end{proposition}


\begin{proposition} Let $r=\rank E_{A,B}(\Q(t))$. Denote with $\overline{A}$ and $\overline{B}$ the images of $A$ and $B$ in $\Q^*/(\Q^*)^6$.
Then
\begin{enumerate}
\item $r\leq 3$.
\item $r=3$ if and only if  $\overline{A}\in \{1,(-3)^3\} $ and $\overline{B}\in \{2^4,(-3)^3 2^4\}$ or $\overline{B}\in \{1,(-3)^3\}$ and $\overline{A}\in \{2^4,(-3)^3 2^4\}$.
\item $r=2$ if and only if $r\neq 3$ and one of
\begin{enumerate}
\item $B$ is a square or $-3$ times a square and $\overline{A}\overline{B}\in \{2^4,-2^43^3\}$;
\item $\overline{A}\in \{2^4 ,2^4(-3)^3 \}$ and $B$ is a cube;
\item $\overline{A},\overline{B} \in \{1,(-3)^3\}$;
\item One of the above with $A$ and $B$ interchanged;
\end{enumerate}
holds.
\end{enumerate}
\end{proposition}

\begin{proof} From the previous Proposition it follows that $r\leq 4$. If $r=4$ then $r_i=1$ for $i=1,2,3,4$, hence $A,B$ and $4AB$ would be all third powers. This is impossible. Hence $r\leq 3$.

The `if'-part of the second point is straightforward. To prove the `only if'-part, note that  $r=3$ is only possible if  $r_1=r_4=1$. Suppose first that  $r_2=1$ holds then $A$ is a cube and  $A$ or $-3A$ is a square, hence $\overline{A}\in  \{1,(-3)^3\}$, moreover $4B$ is a cube and $B$ or $-3B$ is a square, hence $\overline{B}\in \{2^4,(-3)^32^4\}$. The case $r_1=r_3=r_4=1$ is the same  but with $A$ and $B$ interchanged.

The `if'-part of the third bullet point is straightforward. For the `only if'-part, we start with the case $r_1=r_4=1$. In this case $B$ is a square or $-3$ times a square, and $A$ and $B$ are equal modulo squares or differ by $-3$ times a square.
If $A$ and $B$ are equal modulo squares then $4AB$ is a square. Since it is also a third power we find that $4AB$ is a sixth power, and therefore $\overline{A}\overline{B}=2^4$. If $A$ and $B$ differ by $-3$ modulo squares then $4AB$ is $-3$ times a square. Since this is also a third power we find $4AB$ is $(-3)^3$ modulo sixth powers. Hence $\overline{A}\overline{B}=-2^43^3$.

In the remaining cases we have that at least one of $r_1,r_4$ vanishes. Using the symmetry we may assume that $r_4=0$.

If $r_1=r_2=1$ then $r_4=1$ and we have $r=3$. Hence we can exclude this case.

If $r_1=r_3=1$ then $A$ or $-3A$ is a square and $4A$ and $B$  are cubes. This implies that $\overline{A}\in \{2^4,(-3)^3\}$.  Hence we are in the second case.

If $r_2=r_3=1$ then $A$ is a cube and either $A$ is a square or $-3A$ is a square. Hence $\overline{A}\in \{1,(-3)^3\}$. Similarly, $\overline{B} \in \{1,(-3)^3\}$. Hence we are in the third case.
\end{proof}

\begin{remark} Fix a field $K$ of characteristic zero. Let $E/K(t)$ be the generic fiber of a rational elliptic surface. Fix a minimal Weierstrass equation
\[ y^2=x^3+f(t) x+g(t)\]
for $E/K(t)$. The minimality of this equation implies that  $f$ and $g$ are polynomials of degree at most 4 and 6 respectively.
From the classification of  Mordell--Weil groups of rational elliptic surfaces it follows that the \emph{geometric} Mordell--Weil group $E(\overline{K}(t))$ is generated by  points of the form $(a_2t^2+a_1t+a_0,b_3t^3+b_2t^2+b_1t+b_0)$, with $a_i,b_j \in \overline{K}$, and there are at most $240$ such points, see \cite[Theorem 8.33]{SchShi}. The Galois action  on these points  determines the Galois action on the Mordell--Weil group $E(\overline{K}(t))$. However,  there might be  non-trivial points in $E(K(t))$ even when none of the above mentioned points is fixed by the Galois group. This happens for example when $E=E_{A,B,k,1}$ and $-3A$ is a square. In particular the assumption that $K$ is algebraically closed is crucial for \cite[Theorem 8.33]{SchShi} to hold.

Substituting a generic such point in the Weierstrass equation and then considering the coefficients of $t^6,t^5,\dots, t,1$
 yields 7 polynomial equation in 7 variables. These equations define a zero-dimensional scheme $\Sigma$. Suppose now $K=\Q$. In the generic case the $K$-scheme $\Sigma$ has length 240 and consists of a single point $P$ of degree 240. The Galois group of the Galois-closure of $\Q(P)/\Q$ is  $W(E_8)$. Hence this approach is not feasible for most examples of rational elliptic surfaces over $\Q$. However, there are a few well-known exceptions.
 \begin{enumerate}
\item
 If there are few singular fibers then the length of $\Sigma_{\red}$ might be smaller than 240, and this scheme tend to be easier to study. E.g., for $E_{A,B,k,1}$ with $k\in \{1,2,3,4\}$ we obtain that $\Sigma_{\red}$ has length 6. In particular, for $k=2$  we obtain that $\Sigma_{\red}$ is defined by the equations  $a_1^3+A=-b_1^2+B=a_2=a_1=a_0=b_1=b_0=0$.  
 
 \item If the discriminant has a rational zero then one can parametrize the corresponding singular fiber. In this case several of the 7 equations turn out to be linear in one of the variables. See, e.g., \cite[Section 4]{KloKuw}.

\item In certain examples  $\Sigma$ decomposes as a union of several subschemes. This happens for example when $E$ admits an extra automorphism (as in the above example) or there is another rational elliptic surface, which admits a base change which is birational to the surface under consideration. 
\end{enumerate}

 Desjardins and Naskr\k{e}cki show that  for $E_{A,B}$, i.e.,  $f=0, g(t)=At^6+B$, one obtains the third case. They decompose $\Sigma$  as a union of several subschemes. For example they identify schemes coming from  $E_{A,B,k,1}$ for $k=1,2,3,4$. However, working with $E_{A,B}/K(t)$ directly is less advantageous since  there are also redundant subschemes of $\Sigma_{\red}$. E.g., the scheme $\Sigma'$ obtained when studying  $E_{A,B,2,2}(K(t))$ contains the schemes coming from $E_{A,B,1,1}$ and $E_{A,B,4,1}$ and one additional nontrivial subscheme.  However, this latter scheme has a rational  point  only if the former two have such a point. This leads to redundant information, and complicates analysis of $\Sigma$.

It would be interesting to have examples of rational elliptic surfaces with  Mordell--Weil rank 6, 7 or 8 and where the above mentioned approach of  \cite{Bartosz} is practical and where we are not in one of the three exceptional classes. \end{remark}

\bibliographystyle{plain}
\bibliography{remke2}

\end{document}